\documentclass[12pt]{article}

\usepackage{amsfonts}
\usepackage{setspace, amsmath, amssymb, amsthm}
\usepackage[colorlinks = true]{hyperref}
\usepackage[margin = 2.5cm]{geometry}
\usepackage[english]{babel}
\usepackage[utf8x]{inputenc}
\usepackage{amsmath}
\usepackage{graphicx}
\usepackage{listings}
\usepackage{caption}
\usepackage{subcaption}
\usepackage{tikz}
\usepackage{lineno}
\DeclareUnicodeCharacter{2212}{\textendash}

\theoremstyle{plain}
\newtheorem{thm}{Theorem}[section]
\newtheorem{lem}[thm]{Lemma}

\newtheorem{prop}[thm]{Proposition}
\newtheorem{prob}[thm]{Problem}

\theoremstyle{definition}
\newtheorem{ex}{Example}
\newtheorem{rem}[thm]{Remark}

\DeclareMathOperator{\per}{per}
\DeclareMathOperator{\sign}{sgn}

\title{A note on graphs with purely imaginary per-spectrum}

\author{Ranveer Singh\thanks{\texttt{ranveer@iiti.ac.in}, Department of Computer Science and Engineering, Indian Institute of Technology Indore, Indore 453552, India.} \and Hitesh Wankhede\thanks{ \texttt{hiteshwankhede9@gmail.com}, Department of Computer Science and Engineering, Indian Institute of Technology Indore, Indore 453552, India.} \thanks{This author is currently at The Institute of Mathematical Sciences (HBNI), Chennai 600113, India.} }

\begin{document}
        \maketitle
\begin{abstract}
	In 1983, Borowiecki and Jóźwiak posed the problem ``Characterize those graphs which have purely imaginary per-spectrum.'' This problem is still open. The most general result, although a partial solution, was given in 2004 by Yan and Zhang, who show that if $G$ is a bipartite graph containing no subgraph which is an even subdivision of $K_{2,3}$, then it has purely imaginary per-spectrum. Zhang and Li in 2012 proved that such graphs are planar and admit a Pfaffian orientation. In this article, we describe how to construct graphs with purely imaginary per-spectrum having a subgraph which is an even subdivision of $K_{2,3}$ (planar and nonplanar) using coalescence of rooted graphs. 
\end{abstract}

	

\section{Introduction}
We consider simple and undirected graphs. Let $V(G)$ and $E(G)$ denote the vertex set and the edge set of a graph $G$, respectively. If two vertices $i, j \in V (G)$ are adjacent, then we write $i \sim j$. The \emph{adjacency matrix} $A(G)=(a_{ij})$ of a graph $G$ with $V (G) = \{1,\dots,n\}$ is the $n \times n$ matrix in which $a_{ij} = 1$ if $i \sim j$, and $0$ otherwise. The \emph{determinant} and the \emph{permanent} of $A(G)$, are defined as 
$$\det (A(G))=\sum_{\sigma \in S_n}\sign(\sigma)\prod_{i=1}^{n}a_{i,\sigma(i)},$$ 
$$\per(A(G))=\sum_{\sigma \in S_n}\prod_{i=1}^{n}a_{i,\sigma(i)},$$ respectively, where $S_n$ is the set of all permutation of the set $\{1,2,\dots, n\}$ and $\sign(\sigma)$ is the signature of the permutation $\sigma$. The \emph{characteristic polynomial} and the \emph{permanental polynomial} of graph $G$ are defined as $$\phi(G,x)=\det(xI-A(G)),$$ $$\pi(G,x)=\per(xI-A(G)),$$ respectively, where $I$ is the identity matrix of order $n$. The set of all roots (counting multiplicities) of $\phi(G,x)$ and $\pi(G,x)$ is called the \emph{spectrum} and the \emph{per-spectrum} of $G$, respectively. Let $\sigma(G)$ and $\sigma_p(G)$ denote the spectrum and per-spectrum of a graph $G$, respectively. Two graphs are \emph{cospectral} if they have the same spectrum and \emph{per-cospectral} if they have the same per-spectrum. Both the spectrum and the per-spectrum are graph invariants. Because if two graphs are isomorphic, then they must necessarily have the same spectrum and the per-spectrum. The converse need not be true. 

A graph is said to be \emph{determined by its spectrum (per-spectrum)}, if any other graph which is cospectral (per-cospectral) to it is also isomorphic. Haemers and van Dam have conjectured that `Almost all graphs are determined by their spectrum' \cite{haemers_conj, vandam_haemers}. Much work has been done on studying graphs using the characteristic polynomial and its roots in Spectral Graph Theory, but very less using the permanental polynomial in comparison \cite{2016permanental}. The eigenvalues of a matrix are the roots of its characteristic polynomial, but the familiar Linear Algebra is not much helpful when studying the permanent as it is a combinatorial function. However, much work has been done on finding bounds on the permanent \cite{bapat2007recent, zhang2016update}. Further, the determinant and the permanent look quite similar by definition, but computationally they are extremely different. The determinant can be computed in polynomial time using LUP decomposition, while computing permanent is a $\#$P-complete problem \cite{valiant1979complexity}. The `Computation of Determinant vs. Permanent' is a significant problem in Computational Complexity Theory \cite{agrawal2006determinant, banpermanent}. This computational dichotomy is also explored using the general notion of immanants in \cite{curticapean2021full}. Some numerical evidence suggests that the permanental polynomial might be better than the characteristic polynomial in distinguishing graphs \cite{ dehmer2017highly, liu2014enumeration}. This tradeoff between computational hardness and the graph characterizing power is an interesting area to study. For an excellent survey on permanental polynomial, we refer to \cite{2016permanental}. For some recent work on permanental polynomial, we refer to \cite{zhang2012computing, li2016graphs, li2016skew, li2018graphs, li2021matching, wu2018constructing}. 

Let $K_{m,n}$ denote the \emph{complete bipartite graph} with partite sets of size $m$ and $n$, respectively. In case $m=1$, we get a \emph{star graph} with $n$ pendent vertices. Let $K_n$, $P_n$, and $C_n$ denote the \emph{complete graph}, the \emph{path graph}, and the \emph{cycle graph} on $n$ vertices, respectively. Note that the spectrum is always real, but the per-spectrum may not be real. Sachs (1978) had noted that if $G$ is a tree such that $\sigma(G)=\{\lambda_1,\lambda_2,\dots,\lambda_n\}$, then $\sigma_p(G)=\{i\lambda_1,i\lambda_2,\dots,i\lambda_n\}$. Borowiecki (1985) \cite{borowiecki1985spectrum} generalized this observation by showing that a bipartite graph $G$ contains no $C_{k}$ for any $k \equiv 0$ (mod $4$) if and only if $\sigma(G)=\{\lambda_1,\lambda_2,\dots,\lambda_n\}$ and $\sigma_p(G)=\{i\lambda_1,i\lambda_2,\dots,i\lambda_n\}$. Thus, if $G_1$ and $G_2$ are two bipartite graphs without $C_{k}$ for any $k \equiv 0$ (mod $4$), then $G_1$ and $G_2$ are cospectral if and only if they are per-cospectral. In other words, the permanental polynomial is not any more useful than the characteristic polynomial in characterization of such graphs. Earlier, Borowiecki and Jóźwiak (1983) had posed the following open problem. 

\begin{prob}\label{prob:char_imag_perspec}\cite{borowiecki1983note}
    Characterize those graphs which have purely imaginary per-spectrum.\footnote{Zero is also considered to be a purely imaginary number. The per-spectrum is purely imaginary if each of the roots of the permanental polynomial is purely imaginary.}
\end{prob}

Let $\mathcal{G}$ denote the class of graphs with purely imaginary per-spectrum. The \textit{subdivision} of an edge $(i,j)$ is the operation of replacing $(i,j)$ with a path $(i,k,j)$ through a new vertex $k$. Let $G = G_0, G_1, G_2, \dots, G_k$ be a sequence of graphs such that for each $i$, $G_i$ can be obtained from $G_{i−1}$ by subdividing an edge twice. Then, $G_k$ is said to be an even subdivision of $G$. Yan and Zhang (2004) \cite{yan2004permanental} have given the most general but a partial solution to this problem. They showed that if $G$ is a bipartite graph containing no subgraph which is an even subdivision of $K_{2,3}$, then there exists an orientation $G^e$ with skew-adjacency matrix $A(G^e)$ such that $\pi(G, x) = \det(xI − A(G^e))$. The converse was proved by Zhang and Li (2012) \cite{zhang2012computing} along with showing that such graphs are planar and the orientation is Pfaffian. They also give a linear time algorithm to determine whether a given bipartite graph contains no even subdivision of $K_{2,3}$. Since the skew-adjacency matrix $A(G^e)$ is skew-symmetric, it has purely imaginary eigenvalues. It follows that the class of bipartite graphs containing no subgraph of an even subdivision of $K_{2,3}$ is a subset of $\mathcal{G}$. Thus, we have 
\begin{center}
    Class of trees $\subseteq$ Class of bipartite graphs without $C_{k}$ for any $k \equiv 0$ (mod $4$) $\subseteq$ Class of graphs without a subgraph which is an even subdivision of $K_{2,3}$ $\subseteq$ $\mathcal{G}$. 
\end{center}

In this article, we wish to study the following problem. 
\begin{prob}
    Let $G$ be a bipartite graph containing a subgraph which is an even subdivision of $K_{2,3}$. When does it admit purely imaginary per-spectrum? 
\end{prob}

The rest of paper is organised as follows. In Section \ref{sect:prelim}, we recall some definitions, known results, and compute permanental polynomial of rooted trees. In Section \ref{sect:construction}, we describe a construction of graphs with purely imaginary per-spectrum and demonstrate it for some special cases.

\section{Preliminaries} \label{sect:prelim}

The permanental polynomial of a graph $G$ can also be computed using its Sachs subgraphs, that is, subgraphs whose components are either edges or cycles of length at least 3. 

\begin{prop} \cite{borowiecki1980computing} \label{sachs-per}
Let $G$ be a graph on $n$ vertices with  $\pi(G,x)=\sum_{k=0}^{n}b_kx^{n-k}$. Then, $$b_k=(-1)^k\sum_{U_k}2^{c(U_k)},$$
	 where the summation is taken over all Sachs subgraphs $U_k$ on $k$ vertices of $G$, and $c(U_k)$ denotes the number of components of $U_k$ which are cycles. 
\end{prop}  

As a consequence, we get the following. 
\begin{prop} \cite{borowiecki1980computing, borowiecki1985spectrum, spectraofgraphs} \label{prop:bip_coeff}
    A graph $G$ is bipartite if and only if $b_k=0$ for each odd $k\in \{0,1,2,\dots,n\}$
\end{prop}

Next, we show that graphs with purely imaginary per-spectrum are bipartite. 
\begin{prop}\label{prop:gbir}
	If $G \in \mathcal{G}$, then $G$ is bipartite. 
\end{prop}
\begin{proof}
	Let $\sigma_p(G) = \{i\lambda_1,i\lambda_2,\dots,i\lambda_n\}$, then $\pi(G,x) = (x-i\lambda_1)(x-i\lambda_2)\dots(x-i\lambda_n)$. If $k$ is odd, then $b_k$ has the factor $i$. But since $\pi(G,x)$ has all real coefficients, $b_k$ must be zero whenever $k$ is odd. Proof follows from Proposition \ref{prop:bip_coeff}.
\end{proof}

Zhang, Liu and Li give a characterization of bipartite graphs in terms of the per-spectrum.

\begin{thm}\cite{zhang2014note} \label{thm:bip_spec_char}Let $G$ be a graph. Then, 
\begin{enumerate}
\item $\pi(G,x)$ has no negative root.
    \item If all the roots of $\pi(G,x)$ are real, then $G$ is empty.
    \item If $G$ is bipartite, then $\pi(G,x)$ has no real root except $0$.
    \item $G$ is bipartite if and only if the per-spectrum is symmetric with respect to the real and the imaginary axes.
\end{enumerate}
\end{thm}

As a consequence of Proposition \ref{prop:gbir} and Theorem \ref{thm:bip_spec_char}(4), a graph $G \in \mathcal{G}$ if and only if its permanental polynomial has the form $$\pi(G,x)=x^{n-2k}(x^2+c_1)(x^2+c_2)\dots(x^2+c_k),$$ where $c_i>0$ for each $i$ for some $k$. Let $G-S$ denote the induced subgraph on the vertex set $V(G)-S$ for some $S\subseteq V(G)$. In case $S=\{r\}$, we denote it by $G-r$. 

Next, we recall a recursive formula to compute $\pi(G,x)$ of a graph $G$.

\begin{lem}\label{perpoly-vertexrem}\cite{borowiecki1980computing}    
    Let $G$ be a graph, and let $\Gamma_u(G)$ denote the set of cycles passing through a vertex $u \in V(G)$. Then,  $$\pi(G,x) = x\pi(G-u,x) + \sum_{v\sim u}\pi(G-u-v,x) +2\sum_{C\in \Gamma_u(G)} (-1)^{|V(C)|}\pi(G-V(C),x),$$ where $V(C)$ is the set of vertices in the cycle $C$.
\end{lem}

Let $(T,r)$ denote a tree $T$ rooted at the vertex $r$. Let $u_1, u_2,\ldots, u_k$ denote the children of the root $r$, and $(T_1,u_1),(T_2,u_2), \ldots, (T_k,u_k)$ denote the corresponding rooted subtrees for some $k\geq 0$. We write $(T,r) = \mathcal{R}(T_1,T_2,\dots, T_k,r)$. In case $T_1=T_2=\dots=T_k$, let $T':=T_i$ and $u':=u_i$ for each $i=1,\dots k$, then we write $(T,r)=\mathcal{R}(T'^{(k)},r)$. Next, we give a recursive formula to compute permanental polynomial of rooted trees. 
\begin{prop}\label{prop:per-rooted-tree-recur}
     Consider a rooted tree $(T,r) = \mathcal{R}(T_1,T_2,\dots, T_k,r)$. Then,  $$\pi(T,x) = x\pi(T-r,x)+  \sum_{i=1}^{k}\left (\pi(T_i-u_i,x)\prod_{j=1, j\neq i}^{k} \pi(T_j,x) \right ).$$
    In case $(T,r)=\mathcal{R}(T'^{(k)},r)$, then the expression simplifies to $$\pi(T,x) = x\pi(T-r,x) +  k\left (\pi(T'-u',x) \pi(T',x)^{k-1}\right ).$$
\end{prop}
\begin{proof}
    First, note that $T-r =\cup_{i=1}^{k}T_i$. Hence, for each $i=1,\dots k$, we have 
    $$\pi(T-r, x)=\prod_{i=1}^{k} \pi(T_i, x), \text{and } \pi(T-r-u_i,x) = \pi(T_i-u_i,x)\prod_{j=1, j\neq i}^{k} \pi(T_j,x).$$ Using Lemma \ref{perpoly-vertexrem}, we get 
    \begin{align*}
    \pi(T,x)  & = x\pi(T-r,x)+ \sum_{i=1}^{k}\pi(T-r-u_i,x).
    \end{align*}
    Proof follows by substitutions, and the second statement follows as a corollary.  
\end{proof}

The permanental polynomial of star graphs and path graphs can be computed directly using the following proposition. 
\begin{prop}\label{propath}
    \begin{enumerate}
    \item \cite[Proposition 2.10.]{liu2013characterizing} $\pi(K_{1,n},x) = x^{n-1}(x^2+n)$.
    \item \cite[Theorem 3.1.]{liu2013characterizing} $\pi(P_n,x) = \sum_{m=0}^{\lfloor n/2 \rfloor}\binom{n-m}{n-2m}x^{n-2m}.$
\end{enumerate}
\end{prop}
\begin{proof}
 For the star graph, substitute $k=n$ and $T_i=K_1$, and for the path graph, substitute $k=1$ and $T_i=P_{n-1}$ in Proposition \ref{prop:per-rooted-tree-recur}. For the path graph, we get $\pi(P_n,x) = x\pi(P_{n-1},x) + \pi(P_{n-2},x)$ where $\pi(P_0,x)=1$, $\pi(P_1,x)=x$. This recurrence can be solved either using generating series or using Proposition \ref{sachs-per}. 
\end{proof}

We say that a rooted tree $(T,r)$ is \textit{starlike} if $(T,r)=\mathcal{R}(K_{1,l}^{(k)},r)$, and \textit{pathlike} if $(T,r)=\mathcal{R}(P_{l+1}^{(k)},r)$ for some $l\geq 0$. See Figure \ref{fig:starlike}, \ref{fig:pathlike}. Note that for a rooted starlike tree $T$, the subtrees $(K_{1,l},u')$ are rooted at the central vertex of degree $l$. So consequently $\pi(T',x)=x^{l-1}(x^2+l)$, and $\pi(T'-u',x)=\pi(lK_1,x)=x^{l}$, where $lK_1$ denotes $l$ copies of $K_1$. For a rooted pathlike tree $T$, we have $\pi(T',x)=\pi(P_{l+1},x)$, and $\pi(T'-u',x)=\pi(P_{l},x)$. 

\begin{figure}[ht]
     \centering
     \begin{subfigure}[b]{0.30\textwidth}
     \centering
         \begin{tikzpicture}[scale=0.4]
\draw[fill=black] (0,4) circle (3pt);
\draw[fill=black] (0,0) circle (3pt);
\draw[fill=black] (4,0) circle (3pt);
\draw[fill=black] (-4,0) circle (3pt);

\draw[fill=black] (-2.5,-4) circle (3pt);
\draw[fill=black] (-3.5,-4) circle (3pt);
\draw[fill=black] (-4.5,-4) circle (3pt);
\draw[fill=black] (-5.5,-4) circle (3pt);

\draw[fill=black] (2.5,-4) circle (3pt);
\draw[fill=black] (3.5,-4) circle (3pt);
\draw[fill=black] (4.5,-4) circle (3pt);
\draw[fill=black] (5.5,-4) circle (3pt);

\draw[fill=black] (-1.5,-4) circle (3pt);
\draw[fill=black] (-0.5,-4) circle (3pt);
\draw[fill=black] (0.5,-4) circle (3pt);
\draw[fill=black] (1.5,-4) circle (3pt);

\node at (-0.8,4) {$r$};
\node at (-5,0) {$u_1$};
\node at (3,0) {$u_3$};
\node at (-0.8,0) {$u_2$};

\draw[thick] (0,4) -- (0,0);
\draw[thick] (0,4) -- (-4,0);
\draw[thick] (0,4) -- (4,0);
\draw[thick] (4,0) -- (5.5,-4);
\draw[thick] (4,0) -- (4.5,-4);
\draw[thick] (4,0) -- (3.5,-4);
\draw[thick] (4,0) -- (2.5,-4);
\draw[thick] (0,0) -- (1.5,-4);
\draw[thick] (0,0) -- (0.5,-4);
\draw[thick] (0,0) -- (-0.5,-4);
\draw[thick] (0,0) -- (-1.5,-4);
\draw[thick] (-4,0) -- (-2.5,-4);
\draw[thick] (-4,0) -- (-3.5,-4);
\draw[thick] (-4,0) -- (-4.5,-4);
\draw[thick] (-4,0) -- (-5.5,-4);

\end{tikzpicture}
         \caption{Starlike tree $\mathcal{R}(K_{1,4}^{(3)},r)$}
    \label{fig:starlike}
     \end{subfigure}\begin{subfigure}[b]{0.30\textwidth}
     \centering
         \begin{tikzpicture}[scale=0.4]
\draw[fill=black] (0,4) circle (3pt);

\draw[fill=black] (-4,0.5) circle (3pt);
\draw[fill=black] (-4,-0.5) circle (3pt);
\draw[fill=black] (-4,-1.5) circle (3pt);
\draw[fill=black] (-4,-2.5) circle (3pt);
\draw[fill=black] (-4,-3.5) circle (3pt);

\draw[fill=black] (0,0.5) circle (3pt);
\draw[fill=black] (0,-0.5) circle (3pt);
\draw[fill=black] (0,-1.5) circle (3pt);
\draw[fill=black] (0,-2.5) circle (3pt);
\draw[fill=black] (0,-3.5) circle (3pt);

\draw[fill=black] (4,0.5) circle (3pt);
\draw[fill=black] (4,-0.5) circle (3pt);
\draw[fill=black] (4,-1.5) circle (3pt);
\draw[fill=black] (4,-2.5) circle (3pt);
\draw[fill=black] (4,-3.5) circle (3pt);

\node at (-0.8,4) {$r$};
\node at (-5,0.5) {$u_1$};
\node at (3,0.5) {$u_3$};
\node at (-0.8,0.5) {$u_2$};;

\draw[thick] (0,4) -- (0,0.5);
\draw[thick] (0,4) -- (-4,0.5);
\draw[thick] (0,4) -- (4,0.5);

\draw[thick] (4,0.5) -- (4, -0.5) -- (4, -1.5) -- (4, -2.5) -- (4, -3.5);
\draw[thick] (0,0.5) -- (0, -0.5) -- (0, -1.5) -- (0, -2.5) -- (0, -3.5);
\draw[thick] (-4,0.5) -- (-4, -0.5) -- (-4, -1.5) -- (-4, -2.5) -- (-4, -3.5);
\end{tikzpicture}
        \caption{Pathlike tree $\mathcal{R}(P_{5}^{(3)},r)$}
    \label{fig:pathlike}
     \end{subfigure}\begin{subfigure}[b]{0.30\textwidth}
    \centering
    \begin{tikzpicture}[scale=0.5]
\draw[fill=black] (-3,0) circle (3pt);
\draw[fill=black] (-1,0) circle (3pt);
\draw[fill=black] (1,0) circle (3pt);
\draw[fill=black] (3,0) circle (3pt);

\draw[fill=black] (0,-1.5) circle (3pt);

\draw[fill=black] (-1.5,1.5) circle (3pt);
\draw[fill=black] (0,1.5) circle (3pt);
\draw[fill=black] (1.5,1.5) circle (3pt);


\node at (-3.5,0) {$u$};
\node at (3.5,0) {$v$};

\draw[thick] (-3,0) -- (-1,0) -- (1,0) -- (3,0);
\draw[thick] (-3,0) -- (0,-1.5) -- (3,0);
\draw[thick] (-3,0) -- (-1.5,1.5) -- (0,1.5) -- (1.5,1.5) -- (3,0);

\end{tikzpicture}
    \caption{Theta graph $\Theta_{1,2,3}$}
    \label{fig:theta123}
\end{subfigure}

        \caption{}
        \label{incfl}
\end{figure}

Let $u$, $v$ be two vertices, and $a,b,c\geq 1$. A \textit{theta graph} $\Theta_{a,b,c}(u,v)$ is a graph on $a+b+c+2$ vertices such that there are three disjoint paths from $u$ to $v$ of lengths $a+1$, $b+1$ and $c+1$. See Figure \ref{fig:theta123}. Any theta graph is a homeomorphic image of $K_{2,3}$ as it can be obtained by repeated subdivisions of $K_{2,3}$ \footnote{A recursive formula to compute the characteristic polynomial of a theta graph in one of the parameters was given by Sciriha and Fiorini \cite{sciriha_k23}.}. For example, we have $\Theta_{1,1,1}=K_{2,3}$. An even subdivision of $K_{2,3}$ is a theta graph with all odd parameters. Note that after removing a vertex of degree $3$ from a theta graph, we get a pathlike tree. Hence, we can obtain a formula for $\pi(\Theta_{a,b,c},x)$ in terms of the permanental polynomial of path graphs using Lemma \ref{perpoly-vertexrem}.

Consider two rooted graphs $(G_1,r_1)$ and $(G_2,r_2)$. By identifying $r_1$ in $G_1$ with $r_2$ in $G_2$, we construct a new graph $G_1\cdot G_2$ called the \textit{coalescence} of $G_1$ and $G_2$. Schwenk (1974) proved the following relation which holds for both the characteristic and the permanental polynomial. 
\begin{lem}\label{lem:coal-poly-recur}\cite{schwenk1974computing}
    Let $(G_1,r_1)$ and $(G_2,r_2)$ be two rooted graphs, then 
    $$\pi(G_1\cdot G_2,x) = \pi(G_1,x)\pi(G_2-r_2,x) + \pi(G_1-r_1,x)\pi(G_2,x) - x\pi(G_1-r_1,x)\pi(G_2-r_2,x).$$
\end{lem}

The next two results will be helpful in computations. 

\begin{lem}[Folklore]\label{lem:discrim}
	The cubic polynomial $ax^3+bx^2+cx^3+d$ with real coefficients has all its roots real if and only if $\Delta \geq 0$, where $\Delta=18abcd-4b^3d+b^2c^2-4ac^3-27a^2d^2$ is the discriminant. 
\end{lem}

\begin{lem}[Descartes' rule of signs]\label{lem:decartes}
	Let $f(x)$ be an univariate polynomial with real coefficients ordered by descending power of the variable $x$. Then the number of positive roots of $f(x)$ is either equal to the number of sign changes between consecutive (nonzero) coefficients, or is less than it by an even number. Also, the number of negative roots of $f(x)$ is same the number of positive roots of $f(-x)$.
\end{lem}

\section{Construction}\label{sect:construction}

In this section, we describe a construction of graphs with purely imaginary per-spectrum and demonstrate it for some special cases. First, we discuss some examples. 

Bipartite graphs containing no subgraph which is an even subdivision of $K_{2,3}$ are planar, but not all planar graphs are in $\mathcal{G}$. Using Sagemath \cite{sagemath} computations, we observe that $K_{2,3}$ is the smallest planar graph not in $\mathcal{G}$. Let $G_8$ denote the graph on $8$ vertices obtained by subdividing any edge of $K_{2,4}$ twice. This graph $G_8$ is the smallest planar graph in $\mathcal{G}$ containing an even subdivision of $K_{2,3}$. Let $G_{11}$ denote the graph on $11$ vertices obtained as follows: Consider $K_{3,3}$ and $v \in V(K_{3,3})$. Take two edges adjacent to $v$ and subdivide each of them twice. Now add a new vertex $u$ and draw an edge between $u$ and $v$. This graph $G_{11}$ is the smallest nonplanar graph in $\mathcal{G}$ containing an even subdivision of $K_{2,3}$, and it also serves as a motivation for our construction using coalescence. The complete bipartite graph $K_{3,3}$ is the smallest nonplanar graph not in $\mathcal{G}$. See Figures \ref{fig:venn_diag} and \ref{fig:g8_g11} for these graphs. Their permanental polynomial and per-spectrum are as follows. 
\begin{itemize}
    \item $\pi(K_{2,3},x)=x(x^4+6x^2+12)$, \\$\sigma_p(K_{2,3})=\{0,\pm 0.48 \pm 1.80i\}$ (approximated to two decimals),
    \item $\pi(G_8,x)=x^2(x^6+10x^4+33x^2+36)$, \\$\sigma_p(G_8) = \{0, 0, \pm 2i, \pm \sqrt{3}i, \pm \sqrt{3}i \},$
    \item $\pi(K_{3,3},x)=x^6+9x^4+36x^2+36$, \\$\sigma_p(K_{3,3})=\{\pm 1.20i, \pm 0.78 \pm 2.10i\}$ (approximated to two decimals),
    \item $\pi(G_{11},x)=x(x^{10}+14x^8+72x^6+168x^4+172x^2+56)$, \\$\sigma_p(G_{11})=\{0, \pm i\sqrt{4+\sqrt{2}}, \pm i\sqrt{2+\sqrt{2}}, \pm i\sqrt{4-\sqrt{2}}, \pm i\sqrt{2}, \pm i \sqrt{2-\sqrt{2}} \}.$ 
\end{itemize}

\begin{figure}
    \centering 
    \begin{tikzpicture}[scale=0.8]
    \begin{scope}[shift={(3cm,-5cm)}, fill opacity=0.5,
    mytext/.style={text opacity=1,font=\large\bfseries}]
\draw (0,0) ellipse (5cm and 2.5cm);
    \draw (-1.1,0) ellipse (2.5cm and 1.5cm);
    \draw [fill=gray] (1.1,0) ellipse (2.5cm and 1.5cm);
    \draw [fill=white] (0,0) circle (0.6);
\node[mytext] at (0,2.1) (A) {$\mathcal{B}$};
    \node[mytext] at (-2.9,0) (B) {$\mathcal{P}$};
    \node[mytext] at (2.9,0) (C) {$\mathcal{G}$};
    \node[mytext] at (0,0) (D) {$\mathcal{K}$};
    \filldraw[black] (-2,1) circle (2pt) node[anchor=west]{\small $K_{2,3}$};
    \filldraw[black] (0,0.85) circle (2pt) node[anchor=west]{\small $G_8$};
    \filldraw[black] (2,1) circle (2pt) node[anchor=west]{\small $G_{11}$};
    \filldraw[black] (3.6,1) circle (2pt) node[anchor=west]{\small $K_{3,3}$};
\end{scope}
    \end{tikzpicture}
    \caption{$\mathcal{B}$: Bipartite Graphs, $\mathcal{P}$: Planar Bipartite Graphs, $\mathcal{G}$: Graphs with purely imaginary per-spectrum, $\mathcal{K}$: Graphs without a subgraph which is an even subdivision of $K_{2,3}$.}
    \label{fig:venn_diag}
\end{figure}
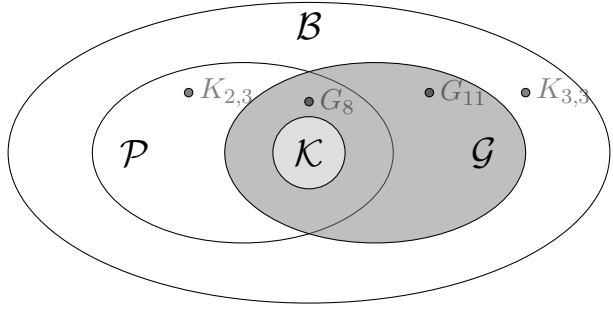

\begin{figure}
    \centering
    \includegraphics[scale=0.4]{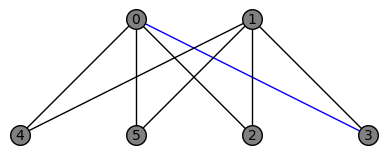}
    \includegraphics[scale=0.4]{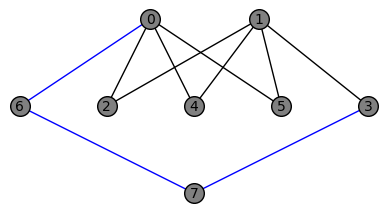}
    \includegraphics[scale=0.4]{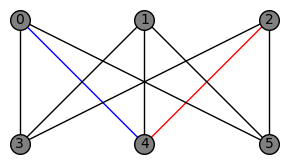}
    \includegraphics[scale=0.4]{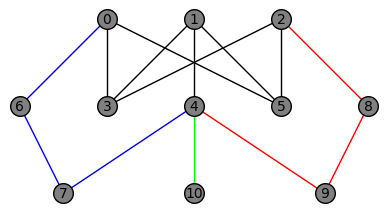}
    \caption{(Left to right) $K_{2,4}$, $G_{8}$, $K_{3,3}$ and $G_{11}$. By subdividing the blue edge twice in $K_{2, 4}$ we obtain $G_8$. By subdividing the blue and the red edge twice in $K_{3, 3}$, and drawing an edge between their common vertex and the newly added vertex, we obtain $G_{11}$.}
    \label{fig:g8_g11}
\end{figure}

Next, we consider the coalescence of a rooted graph and a rooted tree. 

\begin{lem}\label{thm:coal_graph_tree}
    Let $(G_1,r_1)$ be a rooted graph and $(T, r_2)=\mathcal{R}(T'^{(k)},r_2)$ be a rooted tree. Their coalescence $G=G_1\cdot T$ has purely imaginary per-spectrum if and only if the polynomial $$\mathcal{H}(G_1,T):=\pi(G_1,x)\pi(T',x) + k\pi(G_1-r_1,x)\pi(T'-u',x)$$ has all its roots purely imaginary.
\end{lem}
\begin{proof}
Applying Lemma \ref{lem:coal-poly-recur}, we get 
\begin{align*}
    \pi(G,x) &= \pi(G_1,x)\pi(T-r_2,x) + \pi(G_1-r_1,x)\pi(T,x) - x\pi(G_1-r_1,x)\pi(T-r_2,x)\\
    &= \pi(G_1,x)\pi(T-r_2,x) + \pi(G_1-r_1,x)\left (\pi(T,x) - x\pi(T-r_2,x) \right )\\
    &= \pi(G_1,x)\pi(T',x)^k + \pi(G_1-r_1,x)\left (k\pi(T'-u',x)\pi(T',x)^{k-1} \right )\\
    &= \pi(T',x)^{k-1}\left (\pi(G_1,x)\pi(T',x) + k\pi(G_1-r_1,x)\pi(T'-u',x) \right ).
\end{align*}
In the third step, we use Proposition \ref{prop:per-rooted-tree-recur}. We already know that a tree has purely imaginary per-spectrum. Hence, we can ignore the common factor $\pi(T',x)^{k-1}$. 
\end{proof}
In the remaining section, we take different examples of graph $G_1$, in particular, $K_{2,3}$ rooted at a degree $3$ vertex, $K_{2,3}$ rooted at a degree $2$ vertex, and a rooted $K_{3,3}$. As for the tree, we take $T$ to be a rooted starlike or a rooted pathlike tree.

Consider the case when $G_1=K_{2,3}$, and $T$ is starlike. 
\begin{thm}\label{thm:k23-starlike}
The coalescence of $(K_{2,3}, r_1)$ such that $degree(r_1)=3$ and rooted starlike tree $\mathcal{R}(K_{1,l}^{(k)},r_2)$ has purely imaginary per-spectrum if and only if any of the following holds:
\begin{enumerate}
    \item $l+k\geq 4$ when $l\leq 3$,
    \item $k\geq 2l-4$ when $l\geq 4$.
\end{enumerate}
\end{thm}
\begin{proof}
    We have $\pi(K_{2,3},x) = x(x^4+6x^2+12)$ and $\pi(K_{1,3},x) = x^2(x^2+3)$. Now consider the polynomial
\begin{align*}
    \mathcal{H}(G_1,T)& =\pi(G_1,x)\pi(T',x) + k\pi(G_1-r_1,x)\pi(T'-u',x)\\&= \pi(K_{2,3},x)\pi(K_{1,l}, x) + k\pi(K_{1,3},x)\pi(lK_1,x) \\
    &=x(x^4+6x^2+12)x^{l-1}(x^2+l) + kx^2(x^2+3)x^{l} \\
    &=x^l \left (x^6+(l+k+6)x^4+(6l+3k+12)x^2+12l \right ).
\end{align*}
Let $p(x) = x^6+(l+k+6)x^4+(6l+3k+12)x^2+12l$, and $q(x)=-p(i\sqrt{x})=x^3-(l+k+6)x^2+(6l+3k+12)x-12l$. Then, $p(x)$ has all its roots purely imaginary if and only if $q(x)$ has all its roots nonnegative real. 
The discriminant of $q(x)$ is
\begin{align*}
\Delta = 18(−(l+k+6))(6l+3k+12)(−12l)−4(−(l+k+6))^3(−12l)\\+((−(l+k+6))2)(6l+3k+12)^2−4(6l+3k+12)^3−27(−12l)^2.
\end{align*}
It is easy to check that $\Delta\geq 0$ if and only if one of the two conditions in the statement of the theorem holds. By Lemma \ref{lem:discrim}, $q(x)$ has all its roots real if and only if $\Delta\geq 0$. Suppose that $\Delta \geq 0$, and let $\alpha, \beta, \gamma$ be the real roots of $q(x)$. Since $l\geq 0$ and $k\geq 0$, the sign pattern of the coefficients of $q(x)$ is either $[+,-,+,-]$ or $[+,-,+,0]$. Hence, the sign pattern of the coefficients of $q(-x)$ is either $[-,-,-,-]$ or $[-,-,-,0]$. By Lemma \ref{lem:decartes}, the number of negative roots of $q(x)$ is zero. This shows if $\Delta \geq 0$, then all roots of $q(x)$ are nonnegative. The proof follows by Lemma \ref{thm:coal_graph_tree}. 
\end{proof}

\begin{ex}
    Let $k=n$ and $l=0$. Then we get a star $T=K_{1,n}$ is rooted at the central vertex. Let $G_n$ be the coalescence of $K_{2,3}$ (rooted at a vertex of degree $3$) and $K_{1,n}$ (rooted at a vertex of degree $n$), then it has purely imaginary per-spectrum if and only if $n\geq 4$.
\end{ex}

\begin{ex}
    Let $k=1$ and $l=n-1$, then we get a star $T=K_{1,n}$ rooted a pendant vertex. Let $G_n$ be the coalescence of $K_{2,3}$ (rooted at a vertex of degree $3$) and $K_{1,n}$ (rooted at a vertex of degree $1$), then it has all its roots purely imaginary if and only if $n=4$. 
\end{ex}

\begin{rem}\label{rem:k23_3_pathlike}
In case when the rooted tree $T = \mathcal{R}(P_{l+1}^{(k)},r_1)$ is pathlike, substituting $\pi(T',x)=\pi(P_{l+1},x)$ and $\pi(T'-u')=\pi(P_{l},x)$, we get $$\mathcal{H}(G_1,T) = x\left ((x^4+6x^2+12)\pi(P_{l+1},x) + kx(x^2+3)\pi(P_{l},x)\right ).$$ 
Our proof method in Theorem \ref{thm:k23-starlike} depended on reducing the problem to checking real roots of a cubic polynomial. For polynomials of higher degree, there do not appear to be method at hand other than direct computation of roots. For the case $l\leq 7$, this polynomial has all its roots purely imaginary if one of the following holds:
\begin{enumerate}
    \item $k\geq 4$ when $l=0$,
    \item $k\geq 3$ when $l=1,2,3,4$,
    \item $k\geq 13$ when $l=7$.
\end{enumerate}
These conditions are obtained by computations (see Appendix).
\end{rem}
Next, we make a degree $2$ vertex of $K_{2,3}$ a root. 
\begin{thm}
    The coalescence of $(K_{2,3}, r_1)$ such that $degree(r_1)=2$ with a rooted starlike tree $\mathcal{R}(K_{1,l}^{(k)},r_2)$ has purely imaginary per-spectrum if and only if $l+k\geq 4$ and $l\leq 2$.
\end{thm}

\begin{proof}
    We have $\pi(K_{2,3},x) = x(x^4+6x^2+12)$ and $\pi(C_4,x) = (x^2+2)^2$. Now consider the polynomial 
\begin{align*}
    \mathcal{H}(G_1,T)&=\pi(G_1,x)\pi(T',x) + k\pi(G_1-r_1,x)\pi(T'-u',x)\\&= \pi(K_{2,3},x)\pi(K_{1,l}, x) + k\pi(C_4,x)\pi(lK_1,x) \\
    &=x(x^4+6x^2+12)x^{l-1}(x^2+l) + k(x^2+2)^2x^{l} \\
    &=x^l \left ( x^6 +(l+k+6)x^4+(6l+4k+12)x^2+(12l+4k)\right ).
\end{align*}
	By a similar analysis done in Theorem \ref{thm:k23-starlike}, $\mathcal{H}(G_1,T)$ has all its roots purely imaginary if and only if $l+k\geq 4$ and $l\leq 2$. The proof follows by Lemma \ref{thm:coal_graph_tree}. 
\end{proof}

\begin{ex}
    Let $k=n$ and $l=0$, then we get a star $K_{1,n}$ rooted at the central vertex. Let $G_n$ be the coalescence of $K_{2,3}$ (rooted at a vertex of degree $2$) and $K_{1,n}$ (rooted at a vertex of degree $n$). It has purely imaginary per-spectrum if and only if $n\geq 4$.
\end{ex}

\begin{rem}\label{rem:k23_2_pathlike}
    Similarly, in case when the rooted tree $T = \mathcal{R}(P_{l+1}^{(k)},r_1)$ is pathlike, we get $$\mathcal{H}(G_1,T) = x(x^4+6x^2+12)\pi(P_{l+1},x) + k(x^2+2)^2\pi(P_{l},x).$$ 
    For the case $l\leq 6$, this polynomial has all its roots purely imaginary if one of the following holds (see Appendix): 
\begin{enumerate}
    \item $k\geq 4$ when $l=0$,
    \item $k\geq 3$ when $l=1,2$,
    \item $k=3$ when $l=3$,
    \item $k\geq 19$ when $l=5$,
    \item $k\geq 9$ when $l=6$.
\end{enumerate}
\end{rem}

\begin{rem}
    By computations using Sagemath \cite{sagemath}, we observe that $G_8$ is the smallest connected bipartite graph in $\mathcal{G}$ which has an even subdivision of $K_{2,3}$ (see Figure \ref{fig:g8_g11}), and there are no other such graphs on $8$ vertices in $\mathcal{G}$. On $9$ vertices, there are $8$ graphs in $\mathcal{G}$ which contain an even subdivision of $K_{2,3}$ (see Figure \ref{fig:n9_egs_in_g}). Note that $7$ of them can be obtained using the single coalescence of either $K_{2,3}$ or $\Theta_{3,1,1}$ with a starlike or a pathlike tree. The remaining example requires coalescences at two different roots. 

\begin{figure}[ht]
    \centering
    \begin{tabular}{ | c | c| c| } 
  \hline
  The graph $G_1$ & Coalescence of $G_1$ with rooted tree  \\ 
  \hline
  $K_{2,3}$ rooted at a degree $3$ vertex & 
  \includegraphics[scale=0.4]{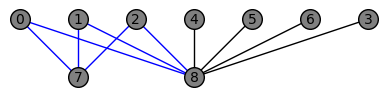}  \includegraphics[scale=0.4]{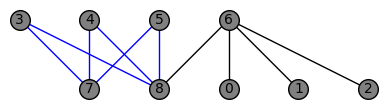}\\
  \hline
  $K_{2,3}$ rooted at a degree $2$ vertex &\includegraphics[scale=0.4]{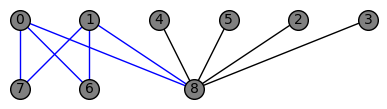} \\
  \hline
  $\Theta_{3,1,1}$ rooted at a degree $3$ vertex&  \includegraphics[scale=0.4]{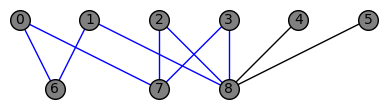}  \includegraphics[scale=0.4]{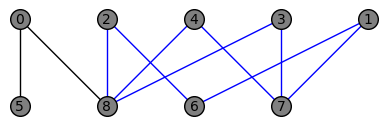}\\
  \hline 
  $\Theta_{3,1,1}$ rooted at a degree $2$ vertex&
    \includegraphics[scale=0.4]{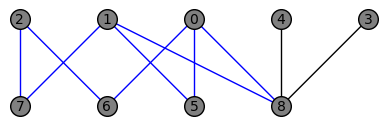} 
    \includegraphics[scale=0.4]{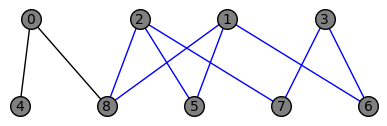}\\
  \hline
  $\Theta_{3,1,1}$ with two root vertices of degree $3$ each &\includegraphics[scale=0.4]{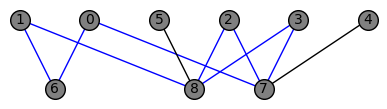} \\ 
  \hline
\end{tabular}
    \caption{Connected bipartite graphs on $9$ vertices in $\mathcal{G}$ which have an even subdivision of $K_{2,3}$ as a subgraph}
    \label{fig:n9_egs_in_g}
\end{figure}

\end{rem}

Next, we produce nonplanar graphs with purely imaginary per-spectrum.  
\begin{rem}\label{k33_starlike}
	Consider the coalescence of $K_{3,3}$ rooted at any vertex with a rooted starlike tree $T = \mathcal{R}(K_{1,l}^{(k)},r_2)$. Then, 
	\begin{align*}
    \mathcal{H}(G_1,T) &= \pi(G_1,x)\pi(T',x) + k\pi(G_1-r_1,x)\pi(T'-u',x) \\
    &= \pi(K_{3,3},x)\pi(K_{1,l}, x) + k\pi(K_{2,3},x)\pi(lK_1,x)\\
    &=(x^6+9x^4+36x^2+36)x^{l-1}(x^2+l) + kx(x^4+6x^2+12)x^{l} \\
    &=x^{l-1} \left ( x^8+ (k+l+9)x^6+(6k+9l+36)x^4+(12k+36l+36)x^2+36l \right ).
\end{align*}
This polynomial has all its roots purely imaginary if $(k,l)\in \{(5,1), (6,0), (7,0) \}$ (see Appendix). In case tree is pathlike $T = \mathcal{R}(P_{l+1}^{(k)},r_2)$, we have 
\begin{align*}
    \mathcal{H}(G_1,T)& =\pi(G_1,x)\pi(T',x) + k\pi(G_1-r_1,x)\pi(T'-u',x)\\&= \pi(K_{3,3},x)\pi(P_{l+1},x) + k\pi(K_{2,3},x)\pi(P_{l},x) \\&=(x^6+9x^4+36x^2+36)\pi(P_{l+1},x) + kx(x^4+6x^2+12)\pi(P_{l},x).
\end{align*}
    This polynomial has all its roots purely imaginary if $(k,l)\in \{ (5,1), (6,0), (7,0) \}$ (see Appendix). Hence, by Lemma \ref{thm:coal_graph_tree}, the coalescence of $K_{3,3}$ rooted at any vertex with a rooted starlike tree $T = \mathcal{R}(K_{1,l}^{(k)},r_2)$ or a rooted pathlike tree $T = \mathcal{R}(P_{l+1}^{(k)},r_2)$ has purely imaginary per-spectrum if $(k,l)\in \{ (5,1), (6,0), (7,0) \}$.
\end{rem}



\section{Conclusion}

There are various ways in which one can make use of the construction idea. The graph $G_1$ in Lemma \ref{thm:coal_graph_tree} can be taken to be any theta graph with odd parameters or any graph containing it. One can also consider multiple coalescences (see \cite{godsil1978new}) or a different definition of the rooted product (see neighbourhood rooted product \cite{godsil1982constructing}). The second graph $T$ need not be only a starlike or a pathlike tree. For example, the coalescence of the cycle $C_6$ and $K_{2,3}$ (rooted at a degree $3$ vertex) has purely imaginary per-spectrum, but it is not obtainable using the construction described in this article. Accounting for such cases is the subject of future study.

\section*{Acknowledgements}
This work is supported by Department of Science and Technology (Govt. of India) through project DST/04/2019/002676. We would also like to thank anonymous referees for suggestions that helped improve this paper. 

\bibliographystyle{plain}
\bibliography{ref}

\section*{Appendix: Computations}\label{sect:comp}
We use Sagemath \cite{sagemath} to examine the roots of the polynomials. 
\lstset{language=Python, mathescape=true,numbers=left,
numberstyle=\tiny, stepnumber=5, numberfirstline=false, firstnumber=1,
basicstyle=\small\ttfamily,columns=fixed, breaklines=true, commentstyle=\textsl,showstringspaces=false}

\begin{lstlisting}
# Create a list of the permanental polynomial of path graphs to be used later
def funct(n):
    a = []
    for i in range(n):
        a.append(add(binomial(i-m,i-2*m)*var("x")^(i-2*m) for m in range(0,floor(i/2)+1)))
    return a 
a = funct(31)

# Define a function that returns 1 if all roots of a given polynomial are purely imaginary
def check(poly):
    roots = poly.roots(x, ring=CC)
    for r in roots:
        if r[0].real() != 0:
            return 0
    return 1
        
# Define the polynomial involved in Remark 3.3.
def poly1(l,k):
    return (var("x")^4+6*var("x")^2+12)*a[l+1] + k*var("x")*(var("x")^2+3)*a[l]
    
# Find all (l,k) pairs such that this polynomial has purely imaginary roots
res1 = []
for l in range(30):
    for k in range(30):
        if check(poly1(l,k)) == 1:
            res1.append((l,k))  
\end{lstlisting}
Similarly, we can define the following polynomials involved in Remark \ref{rem:k23_2_pathlike} and Theorem \ref{k33_starlike}, and examine their roots. 
The output of these programs for $l\leq 10$ and $k\leq 30$ are shown in Figure \ref{fig:scatter}. 
\begin{figure}[ht]
	\centering
	\includegraphics[scale=0.4]{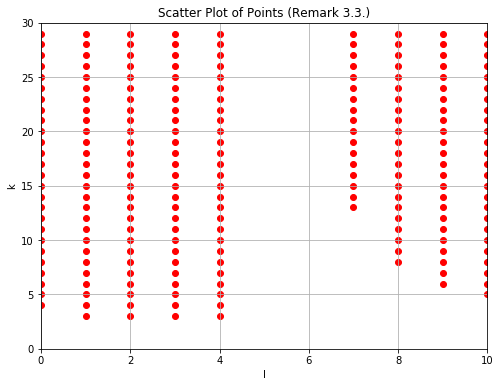}
	\includegraphics[scale=0.4]{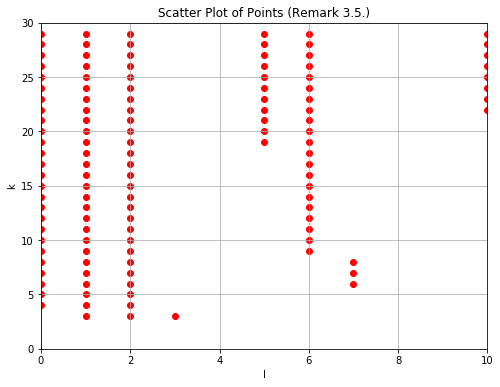}\\
	\includegraphics[scale=0.4]{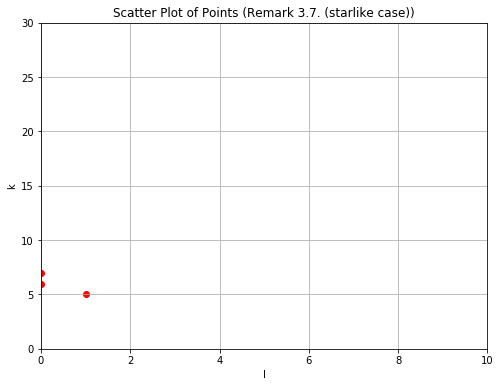}
	\includegraphics[scale=0.4]{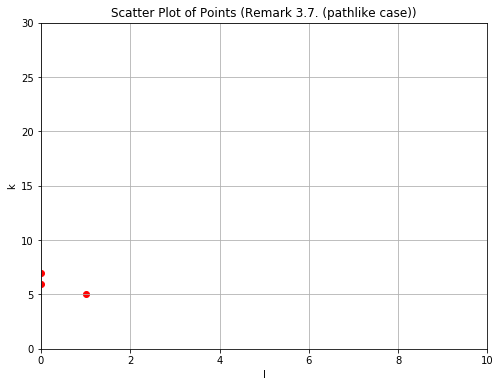}
	\caption{Scatter plot of points $(l,k)$ where the constructed graphs have purely imaginary per-spectrum.}
	\label{fig:scatter}
\end{figure}

\end{document}